\documentclass{amsart}

\usepackage{amsmath,amssymb,amsthm, a4wide}

\usepackage{graphicx,tikz, caption}

\newtheorem{theorem}{Theorem}

\newtheorem{proposition}{Proposition}
\newtheorem{corollary}[theorem]{Corollary}

\newtheorem{lemma}{Lemma}

\theoremstyle{definition}

\theoremstyle{remark}

\DeclareMathOperator{\tr}{tr}

\DeclareMathOperator{\diam}{diam}
\DeclareMathOperator{\inrad}{inrad}
\begin{document}

\title[]{Topological bounds for Fourier coefficients\\ and applications to torsion}

\author[]{Stefan Steinerberger}
\address{Department of Mathematics, Yale University}
\email{stefan.steinerberger@yale.edu}

\begin{abstract} Let $\Omega \subset \mathbb{R}^2$ be a bounded convex domain in the plane and consider 
\begin{align*}
 -\Delta u &=1 \qquad \mbox{in}~\Omega  \\
u &= 0 \qquad \mbox{on}~\partial \Omega.
\end{align*}
If $u$ assumes its maximum in $x_0 \in \Omega$, then the eccentricity of level sets close to the maximum is determined by the
Hessian $D^2u(x_0)$. We prove that $D^2u(x_0)$ is negative definite and give a quantitative bound on the spectral gap
$$ \lambda_{\max}\left(D^2u(x_0)\right)  \leq  - c_1\exp\left( -c_2\frac{\diam(\Omega)}{\inrad(\Omega)} \right) \qquad \mbox{for universal}~c_1,c_2>0.$$
This is sharp up to constants. 
The proof is based on a new lower bound for Fourier coefficients whose proof has a topological component: if $f:\mathbb{T} \rightarrow \mathbb{R}$ is continuous and has $n$ sign changes, then
$$ \sum_{k=0}^{n/2}{  \left| \left\langle f, \sin{kx} \right\rangle \right| +    \left| \left\langle f, \cos{kx} \right\rangle \right|  }   \gtrsim_n   \frac{ | f\|^{n+1}_{L^1(\mathbb{T})}}{  \| f\|^{n }_{L^{\infty}(\mathbb{T})}}.$$
This statement immediately implies estimates on higher derivatives of harmonic functions $u$ in the unit ball: if $u$ is very flat in the origin, then the boundary function $u(\cos{t}, \sin{t}):\mathbb{T} \rightarrow \mathbb{R}$ has to have either large amplitude or many roots. It also implies that the solution of the heat equation starting with $f:\mathbb{T} \rightarrow \mathbb{R}$
cannot decay faster than $\sim\exp(-(\# \mbox{sign changes})^2 t/4)$.
\end{abstract}

\maketitle
\vspace{-20pt}
\section{Introduction}
\subsection{Introduction.} When studying the solution of elliptic equations on planar domains, there is a clear tendency for level sets to become
more elliptical and regular. Different aspects of this phenomenon have been studied for a long time and there are many classical results \cite{bras, caf, caf2, cha, gab, kawohl, lo1, lo2, or,rosay,weinkove};
these results are usually centered around the question whether the level sets are convex or remain convex if the underlying domain is convex. Despite a lot of work,
the subject is still not thoroughly understood. A very simple question is whether the solution of $-\Delta u = f(u)$ on a convex domain $\Omega$ with Dirichlet boundary conditions $\partial \Omega$
always has convex level sets -- this is fairly natural to assume (see e.g. P.-L. Lions \cite{lions}) and was only very recently answered in the negative by Hamel, Nadirashvili and Sire \cite{hamel}.
\vspace{-5pt}
\begin{center}
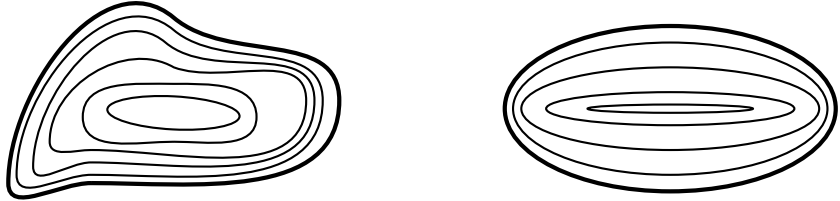
\begin{figure}[h!]
\begin{tikzpicture}[scale=1.1]
\draw[ultra thick] (0,0) to[out=0, in =270] (3,1)  to[out=90, in =320] (1,2) to[out=140, in =90] (-1,0)  to[out=270, in =180] (0,0)    ;
\draw[thick] (0,0.1) to[out=0, in =270] (2.8,1)  to[out=90, in =320] (1,1.85) to[out=140, in =90] (-0.9,0.1)  to[out=270, in =180] (0,0.1)    ;
\draw[thick] (0,0.25) to[out=0, in =270] (2.7,1)  to[out=90, in =320] (0.9,1.7) to[out=140, in =90] (-0.7,0.2)  to[out=270, in =180] (0,0.25)    ;
\draw[thick] (0,0.4) to[out=0, in =270] (2.6,1)  to[out=90, in =340] (1,1.4) to[out=150, in =90] (-0.5,0.5)  to[out=270, in =180] (0,0.4)    ;
\draw[thick] (1,0.5) to[out=0, in =270] (2,0.8)  to[out=90, in =360] (1,1.2) to[out=180, in =90] (-0.1,0.8)  to[out=270, in =180] (1,0.5)    ;
\draw[rotate=-3,thick] (0.95,0.9) ellipse (0.8cm and 0.2cm);
\draw[rotate=0,thick] (7,0.9) ellipse (1cm and 0.05cm);
\draw[rotate=0,thick] (7,0.9) ellipse (1.5cm and 0.2cm);
\draw[rotate=0, ultra thick] (7,0.9) ellipse (2cm and 1cm);
\draw[rotate=0, thick] (7,0.9) ellipse (1.8cm and 0.5cm);
\draw[rotate=0, thick] (7,0.9) ellipse (1.9cm and 0.8cm);
\end{tikzpicture}
\captionsetup{width=0.9\textwidth}
\caption{A cartoon picture of what 'typical level sets' of $-\Delta u = f(u)$ could look like (left) and how they probably will not look like (right).}
\end{figure}
\end{center}

\section{Main Results}
\subsection{The torsion function} We consider, for bounded, convex $\Omega \subset \mathbb{R}^2$, the torsion function
\begin{align*}
 -\Delta u &=1 \qquad \mbox{in}~\Omega  \\
u &= 0 \qquad \mbox{on}~\partial \Omega.
\end{align*}
The torsion function is perhaps \textit{the} classical object in the study of level sets of elliptic equations. Makar-Limanov \cite{makar} proved that  $\sqrt{u}$
is concave on planar convex domains: it follows that level sets are convex and that there is a unique global maximum (see \cite{kawohl, ken} for higher dimensions).
We observe, by simple Taylor expansion, that the eccentricity of the level sets close to the point $x_0$ in which the (unique) maximum is achieved is determined by the eigenvalues of
the Hessian $D^2u(x_0)$. Clearly, every eigenvalue $\lambda$ of $D^2u(x_0)$ satisfies $\lambda \leq 0$. Moreover,
$$ \lambda_1 + \lambda_2 = \tr D^2u(x_0) = \Delta u(x_0) = -1,$$
which means that ellipses can only be highly eccentric if one of the two eigenvalues is close to 0. 
Our main result provides a sharp upper bound on eigenvalues of the Hessian in the maximum.

\begin{theorem} Let $\Omega \subset \mathbb{R}^2$ be a bounded, planar, convex domain and assume the solution of
\begin{align*}
 -\Delta u &=1 \qquad \mbox{in}~\Omega  \\
u &= 0 \qquad \mbox{on}~\partial \Omega.
\end{align*}
 assumes its maximum in $x_0 \in \Omega$. There are universal constants $c_1, c_2>0$ such that 
$$ \lambda_{\max}\left(D^2u(x_0)\right)  \leq  - c_1\exp\left( -c_2\frac{\diam(\Omega)}{\inrad(\Omega)} \right).$$
\label{thm:torsion}
\end{theorem}
This result has the sharp scaling. A simple example (given below) shows that the level sets can be extremely degenerate: close to the maximum, they can look like ellipses with eccentricity that is \textit{exponential} in $\diam(\Omega)/\inrad(\Omega)$ but not more eccentric than that. It is known \cite{keady} that
$$ \log{\left(\frac{1}{4} - \det(D^2 u)\right)} \qquad \mbox{is harmonic}$$
in points satisfying $\det(D^2 u) < 1/4$. This implies that there is nothing particularly special about the value of  $\det(D^2 u(x))$ in $x_0$ and any neighborhood contains points where it is larger and points where it is smaller. Moreover, the determinant will tend
 to 0 in the vicinity of those parts of the boundary that are line segments. On domains with strictly convex boundary, the behavior is much more regular and classical methods
apply; we prove the following simple result.

\begin{proposition} Let $\partial \Omega$ have curvature $\kappa > 0$ and assume the solution of
\begin{align*}
 -\Delta u &=1 \qquad \mbox{in}~\Omega  \\
u &= 0 \qquad \mbox{on}~\partial \Omega.
\end{align*}
 assumes its maximum in $x_0 \in \Omega$. Then, for a universal constant $c>0$,
$$ \lambda_{\max}\left(D^2u(x_0)\right)  \leq - \frac{c}{\inrad(\Omega)^2}\frac{ \min_{\partial \Omega}{\kappa} }{ \max_{\partial \Omega}{\kappa^3}}.$$
\end{proposition}
We do not believe this to be the sharp result but it clearly shows that an exponentially small spectral gap requires parts of the boundary to have very small curvature.  While we are not aware of any result on general convex domains in the literature, the special case of strictly convex domains has been studied by Payne \& Philippin \cite{payne} and our approach to Proposition 1 is inspired by their paper.
One possible application of the result is in the study of Brownian motion: we recall that the torsion function $u(x)$ also describes the expected lifetime of Brownian motion $\omega_x(t)$
started in $x$ until it first touches the boundary. Our results may be understood as a stability property of the maximal lifetime of
Brownian motion on a long, thin convex domain: if one moves away from the point in which lifetime is maximized, then the expected lifetime
in a neighborhood is determined by the eccentricity of the level set.

\subsection{Topological bounds for Fourier coefficients} Somewhat to our surprise, part of our proof turned out to require new results for Fourier series. The first result is so simple that we can prove it in very few lines (more quantitative
versions of the statement are discussed further below). It is not new and well known in the context of Sturm-Liouville theory (see e.g. \cite{sturm}).

\begin{lemma} If $f:\mathbb{T} \rightarrow \mathbb{R}$ is continuous, orthogonal to $\left\{1, \sin{x}, \cos{x}, \sin{2x}, \dots, \cos{nx}\right\}$
and has $2n+2$ roots, then $f$ cannot be orthogonal to both $\sin{(n+1)x}$ and $\cos{(n+1)x}$.
\label{lemma:top}
\end{lemma}
\begin{proof}
We assume the statement fails and start by constructing the conjugate function $Hf$ via the Hilbert transform: if
$$ f(x) = \sum_{k \geq n+2}{a_k \sin{kx} + b_k \cos{kx}}, \quad \mbox{then} \quad  Hf(x) = \sum_{k \geq n+2}{b_k \sin{kx} - a_k \cos{kx}}.$$
The Poisson extension of $f + i Hf$ is holomorphic inside the unit disk. This function has a root of order at least $n+2$ in the origin, which means
that the total number of roots inside the unit disk is at least $n+2$. However, this is also the winding number of $f(e^{it}) + i H f(e^{it})$
around the origin, which implies the existence of at least $2n+4$ real roots of $f$, which is a contradiction.
\end{proof}

The statement is clearly sharp  --  an equivalent way to state it is that any continuous function $f \in C(\mathbb{T})$ in
the orthogonal complement of $\left\{1, \sin{x}, \cos{x}, \sin{2x}, \dots, \cos{nx}\right\}$ has at least $2n + 2$ roots. There is no straightforward way to make it quantitative: if positive and negative $L^1-$mass are close to each
other, then the Poisson extension induces a lot of cancellation, which can make the function very small in the interior. Nonetheless,
assuming some degree of regularity (here: $f$ being bounded) allows to control the extent to which cancellation occurs.

\begin{theorem} \label{thm:fourier} Let $f:\mathbb{T} \rightarrow \mathbb{R}$ be a continuous function having $n$ sign changes. There exists a universal constant $c_n > 0$ such that
$$ \sum_{k=0}^{n/2}{  \left| \left\langle f, \sin{kx} \right\rangle \right| +    \left| \left\langle f, \cos{kx} \right\rangle \right|  }   \geq c_n  \frac{ | f\|^{n+1}_{L^1(\mathbb{T})}}{  \| f\|^{n }_{L^{\infty}(\mathbb{T})}}.$$
\end{theorem}
In particular, multiple roots without a sign change do not affect the scaling.
The proof is using compactness and Lemma \ref{lemma:top} in an essential way: in particular, there is no control over the constant $c_n$. 
Using the Poisson Kernel immediately allows to deduce a completely equivalent formulation of the result for harmonic functions, which is the result we will use in the proof of Theorem 1.
\begin{corollary} Let $u:\mathbb{D} \rightarrow \mathbb{R}$ be harmonic, let $u\big|_{\partial \mathbb{D}}$ be continuous with $n$ sign changes.
Then, for some universal $c_n > 0$,
$$ \sum_{k=0}^{n/2}{  \left| \frac{\partial^k u}{\partial x^k}(0,0)  \right| +   \left| \frac{\partial^k u}{\partial y^k}(0,0) \right|  }   \geq c_n  \frac{ \| u\big|_{\partial \mathbb{D}}\|^{n+1}_{L^1(\mathbb{\partial \mathbb{D}})}}{  \| u\big|_{\partial \mathbb{D}}\|^{n }_{L^{\infty}(\partial \mathbb{D})}}.$$
\end{corollary}
Here, we understand the $0-$th derivative as point evaluation. The statement tells us that if a harmonic function is very flat around the origin, then it either has large amplitude or a large number of roots on the boundary  -- our proof of Theorem 1 requires the statement only for $n=4$.

\subsection{Decay of the heat equation.} Another  interesting application of the result is in the study of lower bounds on the decay of the heat equation on $\mathbb{T}$ for continuous initial datum $f:\mathbb{T} \rightarrow \mathbb{R}$.
The solution of the heat equation $e^{t \Delta} f$ at time $t=1$ is given by the convolution with the heat kernel $\theta_1$, where $\theta_t$ is the Jacobi theta function (see the proof for more details).
 If we take $f(x) = \sin{kx}$, then it is easy to see that
$$ \| \theta_1 *f  \|_{L^2(\mathbb{T})}  \sim e^{-k^2} \qquad \mbox{and} \qquad  \|f\|_{L^2(\mathbb{R})} \sim 1.$$
Clearly, the question is how the $L^2-$norm of $f$ is distributed in the Fourier spectrum and it is not terribly difficult to connect the notion of oscillation with that of decay of the heat equation. A rather easy result, similar to results in \cite{rima, roy} but much simpler, is the following.

\begin{proposition} There exists a universal $c>0$ such that for all differentiable $f:\mathbb{T} \rightarrow \mathbb{R}$ 
$$ \| \theta_1 *f  \|_{L^2(\mathbb{T})}  \gtrsim \exp\left( -c \frac{\|f'\|^4_{L^1(\mathbb{T})}}{\|f\|^4_{L^2(\mathbb{T})}} \right) \|f\|_{L^2(\mathbb{T})}$$
and
$$ \| \theta_1 *f  \|_{L^2(\mathbb{T})}  \gtrsim \exp\left( -c \frac{\|f'\|^2_{L^2(\mathbb{T})}}{\|f\|^2_{L^2(\mathbb{T})}} \right) \|f\|_{L^2(\mathbb{T})}.$$
\end{proposition}
We observe that the second estimate is sharp for $\sin{(kx)}$. One particular consequence of Theorem \ref{thm:fourier} is sharp control depending on the number of roots: as one
would expect, rapid decay of the solution of the heat equation requires the presence of oscillations.
\begin{corollary} Suppose $f:\mathbb{T} \rightarrow \mathbb{R}$ is continuous and changes sign $n$ times. Then, for fixed $t>0$,
$$ \| \theta_t *f  \|_{L^{\infty}(\mathbb{T})}  \gtrsim_{n,t} \frac{ | f\|^{n+1}_{L^1(\mathbb{T})}}{  \| f\|^{n}_{L^{\infty}(\mathbb{T})}}.$$
\end{corollary}
We observe that the implicit constant decays at least like $e^{-n^2 t}$. The Corollary follows immediately from Theorem 2, however, we find it easier to
reverse the arguments, establish Corollary 4 directly and then derive Theorem 2 as a consequence.

\subsection{Outline of the paper.} We discuss several open problems in Section \S 3. The proof of Corollary 4 is given in Section \S 4 and will then imply Theorem 2. Theorem 1 is established in Section \S 5, which also contains a series of explicit computations for the rectangle that show Theorem
1 to be sharp. Section \S 6 gives the proofs of Proposition 1 and 2.

\section{Open problems} 

\textit{1. Convexity of the Domain.} Does Theorem 1 also hold true on domains that are not convex but merely simply connected or perhaps only bounded? The proof uses convexity of the domain $\Omega$ in a very essential way and it is not clear to us whether the statement remains valid in other settings.\\

\textit{2. General elliptic equations.} The question how level sets of elliptic equations behave close to the maximum seems to be of great
intrinsic interest. It would be desirable to have analogous statements for other elliptic equations; one might be inclined to believe that the torsion function behaves in a rather atypical manner and that most
elliptic equations cannot have very eccentric level sets close to their maximum. Consider, for example, the first Laplacian eigenfunction
\begin{align*}
 -\Delta \phi_1 &= \lambda_1 \phi_1 ~~\quad \mbox{in}~\Omega  \\
u &= 0 \qquad \quad \mbox{on} ~ \partial \Omega
\end{align*}
on a rectangle $\Omega = [-D,D] \times [-1,1]$ given by
$$ \phi_1 = \cos{\left( \frac{2}{D \pi} x\right)} \cos{\left(\frac{\pi}{2}x\right)} \qquad \mbox{satisfying} \qquad D^2u(0,0) = \begin{pmatrix} -\frac{4}{\pi^2} D^{-2} & 0 \\ 0 & -\frac{\pi^2}{4} \end{pmatrix}.$$
This suggests $\diam(\Omega)^2\inrad(\Omega))^{-2}$ as the natural candidate for the largest possible eccentricity of level set close to the maximum for this particular equation.  \\

\textit{3. Estimates for Harmonic Functions and Fourier series.} A entirely different family of problems arises naturally from Lemma 1 (and Corollary 3). Let $\mathbb{B}^n$ be the unit ball in $\mathbb{R}^n$ and let $u:\mathbb{B}^n \rightarrow \mathbb{R}$ be harmonic. Is there
an upper bound of the form
$$ \left[\mbox{order of vanishing of}~u~\mbox{in}~0\right]  \lesssim  \mathcal{H}^{n-2}\left( \left\{\|x\| = 1/2: f(x) = 0\right\} \right)^{\alpha},$$
where $\alpha > 0$ is a constant possibly depending on the dimension and $\mathcal{H}^{n-2}$ is the $(n-2)-$dimensional Hausdorff measure? Lemma
1 establishes the result for $n=2$ and $\alpha = 1$ but the topological aspect of our proof does not seem to generalize. It would, of course, also be of great
interest to have versions of Theorem 2 in a more general context and this, too, seems challenging.\\

\section{Proof of Theorem 2}
We start by establishing Theorem 2 since we require it to prove Theorem 1.
One crucial ingredient of the proof will be the heat equation, which is completely explicit: we let
  $$ \theta_t(x) = \frac{1}{2\pi} + \frac{1}{\pi} \sum_{n=1}^{\infty}{e^{- n^2 t}\cos{( n x)}} \qquad \mbox{denote the Jacobi $\theta-$function}$$
 and observe $\theta_t(x) \geq 0$, the symmetry  $\theta_t(x)=\theta_t(2\pi-x)$, 
  $$ \int_{0}^{2\pi}{\theta_t(x)dx} = 1 \qquad \mbox{and} \qquad   \theta_{t}(x) \sim \begin{cases} 1/\sqrt{t} \qquad &\mbox{for}~ |x| \lesssim \sqrt{t} \\ 0 \qquad &\mbox{otherwise.} \end{cases}$$
We also remark that
$$ \left(\frac{\partial}{\partial t} - \Delta\right) (\theta_t * f) = 0,$$
i.e. that $\theta_t * f$ gives the solution of the heat equation with initial datum $f$.  In particular, $\theta_t$ can be interpreted as a Fourier multiplier which implies that it
preserves orthogonality to subspaces. Another crucial property is that for continuous $f:\mathbb{T} \rightarrow \mathbb{R}$ the function $\theta_t *f$ has at most as many
sign changes as $f$. The third important property is that the rapid decay of its Fourier coefficients is a way to obtain compactness of sequences.\\

Before embarking on the proof, we note an elementary observation coming from linear algebra. Let $f: \mathbb{R} \rightarrow \mathbb{R}$ be compactly supported and suppose
that the sum of the number of connected components of $\left\{x:f(x) > 0\right\}$ and $\left\{x:f(x) < 0\right\}$ is $n$.
\begin{lemma} There exists a polynomial of degree $n-1$ such that
$$  \int_{\mathbb{R}}{p(x) f(x) dx}  \neq 0.$$
\end{lemma}
\begin{proof} Let $V$ denote the $n-$dimensional vector space of polynomials of degree $n-1$. We denote the connected components of $\left\{x:f(x) > 0\right\}$ and $\left\{x:f(x) < 0\right\}$ by $I_1, I_2, \dots, I_n$
and construct a linear map $T:V \rightarrow \mathbb{R}^n$ via
$$ T(p) = \left( \int_{I_1}{p(x) f(x) dx}, \int_{I_2}{p(x) f(x) dx}, \dots, \int_{I_n}{p(x) f(x) dx} \right).$$
If the statement were to fail, then $T$ maps $V$ into the $(n-1)-$dimensional subspace
$$ \left\{(x_1, \dots, x_n) \in \mathbb{R}^N: x_1 + x_2 + \dots + x_n =0\right\}$$
and, as a consequence, $\mbox{ker}(T) \neq \emptyset$. Let $0 \neq p \in \mbox{ker}(T)$. We observe that $p$ has to have at least one root inside every interval $I_j$,
which implies that $p$ has $n$ roots and this is a contradiction.
\end{proof}

\subsection{Proof of Corollary 4.} 
\begin{proof} We want to show that, for fixed $t > 0$,
 $$ \| \theta_t *f  \|_{L^{\infty}(\mathbb{T})}  \gtrsim_{n} \frac{ | f\|^{n+1}_{L^1(\mathbb{T})}}{  \| f\|^{n}_{L^{\infty}(\mathbb{T})}}.$$
We assume the statement fails and use the scaling symmetry to extract a subsequence $f_k$ normalized to $\|f_k\|_{L^1(\mathbb{R})} = 1$ and having at most $n$ roots  such that
 $$ \| \theta_t *f_k  \|_{L^{\infty}(\mathbb{T})}  \| f_k\|^{n}_{L^{\infty}(\mathbb{T})}  \rightarrow 0.$$
The normalization $\|f_k\|_{L^1(\mathbb{R})} = 1$ implies $\|f_k\|_{L^{\infty}(\mathbb{R})} \gtrsim 1$, which immediately implies
 $$ \| \theta_t *f_k  \|_{L^{\infty}(\mathbb{T})}  \rightarrow 0.$$
This condition actually implies that the term goes to 0 for all values of $t>0$, which can be seen as follows. Suppose the statement
fails and that $ \| \theta_s *f_k\|_{L^1(\mathbb{T})} \geq \delta > 0$ for infinitely many $n \in \mathbb{N}$ for some $0<s<t$. It is easily seen that
$$  \|\nabla( \theta_s *f_k)\|_{L^2(\mathbb{T})}  \lesssim_s  \| f_k\|_{L^1(\mathbb{T})} = 1.$$
Moreover, by Cauchy-Schwarz
$$  \| \theta_s *f_k\|_{L^2(\mathbb{T})}  \gtrsim \| \theta_s *f_k\|_{L^1(\mathbb{T})} \geq \delta \qquad \mbox{for infinitely many}~k.$$
Expanding into Fourier series (and using orthogonality to constants) gives that for such $k$
$$
 \delta^2 \lesssim \| \theta_s *f_k\|^2_{L^2(\mathbb{T})} = \sum_{j \in \mathbb{Z}}{ \left| \left\langle \theta_s * f_k, e^{i j x}\right\rangle\right|^2 } \leq  \sum_{j \in \mathbb{Z}}{ |j|^2  \left| \left\langle  \theta_s * f_k, e^{ijx}\right\rangle\right|^2 } =  \|\nabla (\theta_s *f_k)\|_{L^2(\mathbb{T})} \lesssim 1,
$$
which implies that at least half of the $L^2-$mass lies at frequencies $|k| \lesssim \delta^{-1}$. Then, however, using the semigroup property
\begin{align*}
 \| \theta_t *f_n\|^2_{L^2(\mathbb{T})} &=  \| \theta_{t-s} * \theta_s *f_k\|^2_{L^2(\mathbb{T})} =  \sum_{j \in \mathbb{Z} }{ e^{-j^2(t-s)} \left| \left\langle \theta_s * f_k, e^{ijx}\right\rangle\right|^2 } \\
&\geq \sum_{j \in \mathbb{Z} \atop |j| \lesssim \delta^{-1} }{ e^{-j^2(t-s)} \left| \left\langle \theta_s * f_k, e^{ijx}\right\rangle\right|^2 } \gtrsim  e^{-\delta^{-2}(t-s)} \sum_{j \in \mathbb{Z} \atop |j| \lesssim \delta^{-1} }{ \left| \left\langle \theta_s * f_k, e^{ijx}\right\rangle\right|^2 }  \\
&\gtrsim e^{-\delta^{-2}(t-s)} \| \theta_s *f_k\|^2_{L^2(\mathbb{T})}  \gtrsim e^{-\delta^{-2}(t-s)} \delta > 0.
\end{align*}
This trivially yields
$$  \| \theta_t *f_k\|^2_{L^{\infty}(\mathbb{T})}  \gtrsim e^{-\delta^{-2}(t-s)} \delta > 0,$$
which then implies
$$  \| \theta_t *f_k\|_{L^{1}(\mathbb{T})}  \gtrsim e^{-\delta^{-2}(t-s)} \delta > 0,$$
and contradicts the assumption $ \| \theta_t *f_k\|_{L^{\infty}(\mathbb{T})} \rightarrow 0$.
We now define, for any such function $f_k$ and any fixed $\delta > 0$, the set that is $\delta-$far away from the roots
$$ A_{\delta} = \left\{ x \in \mathbb{T}: \mbox{distance between}~x~\mbox{and closest root of}~f_k \geq \delta\right\}$$
and obtain that
$$ \lim_{k \rightarrow \infty}{ \int_{A_{\delta}}{|f_k|dx} } = 0$$
because otherwise one would not obtain $ \| \theta_t *f_k\|_{L^{\infty}(\mathbb{T})} \rightarrow 0$ 
(here a suitable $t$ leading to a contradiction would then be on the scale $t \sim \delta^{2+}$).
This implies that the interesting examples, in order to have rapid cancellation, need to have essentially all their $L^1-$mass close to a root. 
We now interpret the roots of $f_k$ as a sequence in $\mathbb{T}^n$. Clearly, there exists a convergent subsequence, which
we again denote by $f_k$, for which the roots also converge. There are now a variety of possible scenarios: all roots could remain separated, two of the roots could converge
to the same point while all other roots remain separated, all roots could converge on the same point or any other number of scenarios. We will only deal with the case where
all roots converge to the same point, our argument immediately extends to all other cases and also shows that all the other cases cannot yield the sharp scaling. 
We have already seen that all the $L^1-$mass has to move towards the limit point of the roots.
We observe that if 
$$ \| \theta_t *f_k  \|_{L^{\infty}(\mathbb{T})}  \| f_k\|^{n}_{L^{\infty}(\mathbb{T})}  \rightarrow 0,$$
then we also have
$$ \| \theta_{s} *f_k  \|_{L^2(\mathbb{T})}  \| f_k\|^{n}_{L^{\infty}(\mathbb{T})}  \rightarrow 0 \qquad \mbox{for all}~s \geq t.$$
 This allows us to take the entire set $\theta_s$ for, say, $t < s < 2t$
and all its translations and use the triangle inequality to conclude that for all functions
$$g \in \left\{ \sum_{j = 1}^{N}{\alpha_j \theta_{s_j}(x - \beta_j)}: \quad N \in \mathbb{N}, \alpha_j \in \mathbb{R}, \beta_j \in \mathbb{T} \right\}$$
we necessarily also have $ \| g *f_k  \|_{L^2(\mathbb{T})}  \| f_k\|^{n}_{L^{\infty}(\mathbb{T})}  \rightarrow 0$. We note that this set is an infinite-dimensional subspace
and simple expansion in Fourier series shows that for $0< s_1 < s_2 < \dots < s_k < \infty$, the functions $\theta_{s_1}, \theta_{s_2}, \dots, \theta_{s_k}$ 
are, when put together with their first $k-1$ derivative, linearly independent (in the sense of the Wronskian). $f_k *g$ being uniformly small for all $g$ means that $\left\langle f_k, g(\cdot - \beta)\right\rangle$
is uniformly small: this is only possible if there are local oscillations on a small scale that use the high degree of smoothness of $g$ to induce cancellations.
At the same time the functions $f_k$ cannot be too concentrated, the normalization $\|f_k\|_{L^1(\mathbb{R})} = 1$ forces the roots to be on scale
$\gtrsim 1/\|f_k\|_{L^{\infty}(\mathbb{R})}$. A simple Taylor expansion of $g$ in a point and Lemma 4 imply that the leading contribution will come from one of the first $n$ derivatives since
$f_{n_k}$ cannot vanish on all polynomials of degree $n$ and this implies the result: higher derivatives yield smaller contributions and by taking $g$ from the family to have suitable derivatives, we see  
 $$ \| g *f_k  \|_{L^{\infty}(\mathbb{T})} \gtrsim (\mbox{scale})^{n} = \|f_k\|^{-n}_{L^{\infty}(\mathbb{R})},$$
which gives the result.
\end{proof}

The proof, as a byproduct, allows us to construct examples showing that the result is sharp and suggests more refined statements. The construction is inspired by the second-order differential quotient
$$ g''(x) \sim \frac{g(x) - 2g(x+\varepsilon) +  g(x+2\varepsilon)}{\varepsilon^2}.$$
We define a function $f:\mathbb{T} \rightarrow \mathbb{R}$ by
$$ f(\theta) = \begin{cases} 1 \qquad &\mbox{for}~0 \leq x \leq \varepsilon \\
-2 \qquad &\mbox{for}~\varepsilon < x \leq 2\varepsilon \\
1 \qquad &\mbox{for}~2\varepsilon < x \leq 3\varepsilon \\
0 \qquad &\mbox{otherwise}.
\end{cases}$$
$f$ satisfies $\|f\|_{L^{\infty}(\mathbb{T})} = 2$,  $\|f\|_{L^1(\mathbb{T})} = 4 \varepsilon$, it changes sign exactly 2 times and has
$$ \| \theta_t * f \|_{L^{\infty}(\mathbb{T})} \sim_t \varepsilon^3   \sim \| f \|^3_{L^1(\mathbb{T})}.$$
The coefficients in more general numerical differentiation schemes immediately give a construction for the more general case.
The proof allows for improved statements: as long as not too many roots cluster in a certain area, the decay cannot be faster than 
$\| f\|^{k+1}_{L^1(\mathbb{T})} /  \| f\|^{k}_{L^{\infty}(\mathbb{T})},$ where $k$ is the maximal number of roots that cluster in a point.
Moreover, whenever these inequalities are close to being attained, then $f$ has to be close to a numerical differentiation scheme that
vanishes on the first few derivatives.

\subsection{Proof of Theorem 2}
\begin{proof} The proof proceeds by contradiction and compactness. The scaling symmetry allows us to assume that $\|f\|_{L^1(\mathbb{T})} = 1$ and it suffices to show that then
$$ \| f\|^{n}_{L^{\infty}(\mathbb{T})}  \left(  \sum_{k=0}^{n/2}{  \left| \left\langle f, \sin{kx} \right\rangle \right| +    \left| \left\langle f, \cos{kx} \right\rangle \right|  }   \right) \gtrsim_n 1 .$$
We assume now that there exists a sequence $f_{\ell}$, all elements of which are normalized $\|f_{\ell}\|_{L^1(\mathbb{T})} = 1$ and have at most $n$ roots such that
$$ \| f_{\ell} \|^{n}_{L^{\infty}(\mathbb{T})}   \left(  \sum_{k=0}^{n/2}{  \left| \left\langle f_{{\ell}}, \sin{kx} \right\rangle \right| +    \left| \left\langle f_{{\ell}}, \cos{kx} \right\rangle \right|  }   \right)   \rightarrow 0.$$
The argument proceeds in very much the same spirit as the argument above. We first argue that all the $L^1-$mass has to move towards the roots: if this were not the case, then for some sufficiently
small $t>0$ the subsequence $\theta_t * f_{\ell}$ would satisfy $\|\theta_t * f_{\ell}\|_{L^1} \gtrsim 1$ and therefore exhibit a convergent subsequence $\theta_t * f_{\ell_k} \rightarrow \phi$. The limiting function $\phi$
cannot have more than $n$ roots. However, since 
\begin{align*}
 \sum_{k=0}^{n/2}{  \left| \left\langle f_{{\ell}}, \sin{kx} \right\rangle \right| +    \left| \left\langle f_{{\ell}}, \cos{kx} \right\rangle \right|  }   &\sim_n   \left(\sum_{k=0}^{n/2}{  \left| \left\langle f_{{\ell}}, \sin{kx} \right\rangle \right|^2 +    \left| \left\langle f_{{\ell}}, \cos{kx} \right\rangle \right|^2  } \right)^{1/2} \\
&\geq \left(\sum_{k=0}^{n/2}{  \left| \left\langle \theta_t* f_{{\ell}}, \sin{kx} \right\rangle \right|^2 +    \left| \left\langle \theta_t *f_{{\ell}}, \cos{kx} \right\rangle \right|^2  } \right)^{1/2} \\
&\sim_n \sum_{k=0}^{n/2}{  \left| \left\langle \theta_t*f_{{\ell}}, \sin{kx} \right\rangle \right| +    \left| \left\langle \theta_t* f_{{\ell}}, \cos{kx} \right\rangle \right|  },
\end{align*}
we can conclude that 
$$\sum_{k=0}^{n/2}{  \left| \left\langle\phi, \sin{kx} \right\rangle \right| +    \left| \left\langle \phi, \cos{kx} \right\rangle \right|  } =0$$
which is a contradiction. Therefore, all the $L^1-$mass has to move towards the roots. As before, we consider the sequence of roots as a sequence on $\mathbb{T}^n$, extract a convergent
subsequence and focus on the largest number of roots converging in a point (with nontrivial $L^1-$mass concentrating there). The rotational invariance of Fourier subspaces allows us to
assume that this occurs in $x=0$. A simple Taylor expansion of $\sin{kx}$ and $\cos{kx}$ then recovers classical polynomials and Lemma 2 applies exactly in the same way as before.
Alternatively, we can fix $t > 0$ and consider $\theta_t * f_{\ell}$. We can implicitly define sequence $\lambda_{\ell}$ via $\| \lambda_{\ell} \theta_t * f_{\ell}\|_{L^{\infty}(\mathbb{T})} = 1$.
Corollary 4 implies that $\lambda_{\ell} \lesssim \|f_{\ell}\|_{L^{\infty}(\mathbb{T})}^n$. The contradiction then follows from extracting a convergent subsequence of $\lambda_{\ell} \theta_t * f_{\ell}$.
\end{proof}

It is easily seen that the diffusive nature of the heat equation, and the particular consequence of not increasing the number of roots, is crucial. Nonetheless, this type
of idea seems to have been used rarely, one particularly striking instance being the work of Eremenko \& Novikov \cite{erem1, erem2}. The author is grateful to Alexander
Volberg for informing him about these results.

\section{Proof of Theorem 1}

\subsection{An explicit computation for rectangles.} In this section, we quickly record the torsion function $v$ on the rectangle $[-a, a] \times [-b,b] \subset \mathbb{R}^2$. We think of $a$ as the 'long' side and $b \sim 1$ as the short side. A fairly standard way to solve the problem is to start with an ansatz 
$b^2 - y^2$ solving the equation in the interior and then adding a harmonic function with correcting boundary condition, which can be done in closed form because
the rectangular domain allows for the use of Fourier series (we refer to standard textbooks,
e.g. \cite{saad}, for details). The final expression is
$$ v(x,y) = (b^2 - y^2) - \frac{32 b^2}{\pi^3} \sum_{n \geq 1 \atop n~odd}{   \frac{(-1)^{(n-1)/2}}{ n^3 \cosh{ \frac{n \pi a}{2 b}}} \cos{\frac{n \pi y}{2b}} \cosh{\frac{n \pi x}{2b}}}.$$
It is easy to see that $\Delta v = -1$ and that $v(x, \pm b) = 0$. It remains to check the boundary conditions $v(\pm a, y)$. There, however,
$$ b^2 - y^2 = \frac{32 b^2}{\pi^3} \sum_{n \geq 1 \atop n~odd}{   \frac{(-1)^{(n-1)/2}}{ n^3} \cos{\frac{n \pi y}{2b}}} \qquad \mbox{for}~-b \leq y \leq b$$
is merely the standard Fourier series. We also compute, for future reference, 
$$ \frac{\partial^2 v}{\partial x^2}(0,0) =  - \frac{8}{\pi} \sum_{n \geq 1 \atop n~odd}{   \frac{(-1)^{(n-1)/2}}{ n \cosh{ \frac{n \pi a}{2 b}}} }.$$
We note that this is an alternating series and, using $\cosh{x} = (e^x + e^{-x})/2 \leq e^x$ for $x > 0$,
$$ \frac{\partial^2 v}{\partial x^2}(0,0) \leq - \frac{8}{\pi} \left( \frac{1}{  \cosh{ \frac{ \pi a}{2 b}}}  -  \frac{4}{  \cosh{ 2\frac{ \pi a}{2 b}}} \right) \leq - \frac{4}{\pi}  \frac{1}{  \cosh{ \frac{ \pi a}{2 b}}}  \leq -\frac{4}{\pi} \exp\left(-\frac{\pi}{2} \frac{a}{b}\right)
$$
which we can also phrase as 
$$ \frac{\partial^2 v}{\partial x^2}(0,0) \leq - \frac{4}{\pi} \exp \left(  -\frac{\pi}{2} \frac{a}{b} \right).
$$
Note that, by merely taking the first term of the alternating series, we get that
$$ \frac{\partial^2 v}{\partial x^2}(0,0) \geq - \frac{8}{\pi} \frac{1}{  \cosh{ \frac{ \pi a}{2 b}}} \geq -\frac{16}{\pi} \exp \left(  -\frac{\pi}{2} \frac{a}{b} \right),$$
which shows that our main result is sharp. We also remark that 
$$ \frac{\partial^2 v}{\partial x^2} +  \frac{\partial^2 v}{\partial y^2} = \Delta v = -1$$
and thus
$$ -1 +   \frac{1}{2\pi} \exp \left(  -\frac{\pi}{2} \frac{a}{b} \right) \leq  \frac{\partial^2 v}{\partial y^2}(0,0) \leq -1 +\frac{2}{\pi} \exp \left(  -\frac{\pi}{2} \frac{a}{b} \right).$$

\subsection{Order of Vanishing.} We first establish a simple estimate that generalizes Lemma \ref{lemma:top} to simply connected domains
in the plane. The actual argument is not substantially different and Riemann mapping would serve as an alternative way to approach the statement
by reducing it to the case of the unit disk. We consider it instructive, however, to see how the result could be approached from the perspective of partial differential
equations (in particular, no complexification occurs). We emphasize again that the result is very likely to be already known.

\begin{lemma} \label{simple} Let $\Omega \subset \mathbb{R}^2$ be a simply connected domain, let $f:\partial \Omega \rightarrow \mathbb{R}$ be continuous
and denote
\begin{align*}
 \Delta u &=0 \qquad \mbox{in}~\Omega  \\
u &= f \qquad \mbox{on}~\partial \Omega.
\end{align*}
Then, for any $x_0 \in \Omega$,
$$ \left[\mbox{order of vanishing of}~u~\mbox{in}~x_0\right] + 1 \leq  \frac{1}{2} \#\left\{y \in \partial \Omega: f(y) = 0\right\}.$$
\end{lemma}
\begin{proof} Assume $u$ vanishes up to order $m \in \mathbb{N}$ in $x_0 \in \Omega$. This means that the function vanishes together
with derivatives up to some derivative of order $m+1$ that does not vanish. It is classical (see e.g. Bers \cite{bers}) that locally in a neighborhood of $x_0$, i.e.
for all $x \in B(x_0, \varepsilon)$
$$ u(x) = \sum_{k=m}{a_k p_k(x-x_0)} + \mbox{l.o.t.}$$
for harmonic polynomials $p_k$ starting at order $m$ and coefficients $a_k \in \mathbb{R}$. Harmonic polynomials in the plane
are merely given by rotations of
$$ p(r,\theta) = r^k \cos{(k \theta)},$$
which restricts the structure of the nodal line: locally, the zero set is given by $2m$ curves that partition the neighborhood into
$2m$ domains. 
\begin{center}
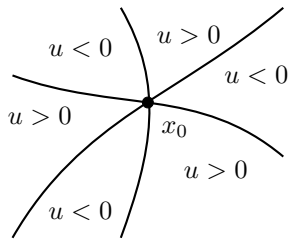
\begin{figure}[h!]
\begin{tikzpicture}[scale = 1.8]
\filldraw (0,0) circle (0.04cm);
\node at (0.2,-0.2) {$x_0$};
\draw [thick] (-1,-1)  to[out=60, in=220] (1,0.7);
\draw [thick] (-1,0.2)  to[out=340, in=140] (1,-0.4);
\draw [thick] (-0.2,0.7)  to[out=300, in=70] (-0.2,-1);
\node at (0.3, 0.5) {$u>0$};
\node at (-0.5, 0.4) {$u<0$};
\node at (-0.8, -0.1) {$u>0$};
\node at (0.8, 0.2) {$u<0$};
\node at (-0.5, -0.8) {$u<0$};
\node at (0.5, -0.5) {$u>0$};
\end{tikzpicture}
\caption{The set $\left\{x:u(x) = 0\right\}$ near $x_0$ if $u$ vanishes in $x_0$ together with gradient and Hessian but has a nonvanishing third derivative.}
\end{figure}
\end{center}
We conclude by noting that every single one of these lines has to touch the boundary because no two such lines can meet: if they did, they would enclose a domain. Since
$u$ is harmonic and vanishes on the boundary of the domain, it would be identically 0 in the inside, which is a contradiction.
on which the function $u$ is harmonic. This gives rise
to $2m$ roots of the boundary function $f$. 
\end{proof}
The usual harmonic polynomials on the disk show the result to be sharp.
We will use the following special case: if $f$ has exactly 4 roots and $u(x_0) = 0$, then $u$ can vanish simultaneously with $\nabla u$
but not all second derivatives vanish.

\subsection{Proof of Theorem 1.}

\begin{proof} Let $\Omega \subset \mathbb{R}^2$ be a bounded, planar, convex domain with diameter $\diam(\Omega)$, inradius $\inrad(\Omega)$ and assume that
\begin{align*}
 -\Delta u &=1 \qquad \mbox{in}~\Omega  \\
u &= 0 \qquad \mbox{on}~\partial \Omega
\end{align*}
has its global maximum in $x_0 \in \Omega$. The result of Makar-Limanov \cite{makar} implies that $x_0$ is unique. There is a scaling symmetry for solutions on rescaled domains $\Omega \rightarrow \lambda \Omega$
$$ u(x,y) \rightarrow \lambda^2 u\left( \frac{x}{\lambda}, \frac{y}{\lambda}\right) \qquad \mbox{for}~\lambda > 0,$$
which keeps the Hessian invariant and allows by scaling to assume that the domain satisfies $\inrad(\Omega)=1$. By rotating and translating the domain
$\Omega$, we may assume that $x_0 = (0,0)$ and that $D^2u(x_0)$ has the form
$$ D^2u(x_0) = \begin{pmatrix} \lambda_0 & 0 \\ 0 & 1 - \lambda_0 \end{pmatrix}$$
and that $|\lambda_0| \ll 1$ is the eigenvalue close to 0. The goal is now to bound this quantity from above. Note that $\inrad(\Omega) = 1$ and convexity of $\Omega$ imply $u(x_0) = u(0,0) \sim 1$ up to absolute constants. We will now consider the rectangle
$$R = [-10\diam(\Omega), 10\diam(\Omega)] \times [-a,a] \subset \mathbb{R}^2$$ where $a$ is an unknown to be determined implicitly; we also consider the torsion function on the rectangle
as
\begin{align*}
 -\Delta v &=1 \qquad \mbox{in}~R  \\
v &= 0 \qquad \mbox{on}~\partial R.
\end{align*}
Standard estimates imply that $v(0,0) \sim a^2$ as long as $a \leq \diam(\Omega)$ and that $v(0,0)$ is monotonically increasing in $a$. This allows us to implicitly define $a^*$
as the unique value for which $v(0,0) = u(0,0)$ and we henceforth use $R$ to denote the rectangle with that side length. We note that if $\Omega_1 \subsetneq \Omega_2$, then the maximum of the torsion function on $\Omega_2$ is larger than the
maximum of the torsion function on $\Omega_1$. This implies that $\Omega$ does not contain $R$ -- the convexity of $\Omega$ and the fact that $R$ is comprised of 4 straight line
segments (2 of which guaranteed to be outside of $\Omega$) poses severe geometric restrictions on the setup.
 
\begin{center}
\begin{figure}[h!]
\begin{tikzpicture}[scale = 1]
\filldraw (0,0) circle (0.04cm);
\node at (0.5,-0.3) {$x_0=(0,0)$};
\draw [ thick] (-4, 1) -- (4,1.0001);
\draw [ thick] (-4, -1) -- (4,-1.001);
\draw [ thick] (-4, -1) -- (-4,1.0001);
\draw [ thick] (4, -1) -- (4,1.0001);
\node at (4.2, -1.1) {$R$};
\node at (0, 1.2) {$\Omega \setminus R$};
\node at (0, 0.6) {$\Omega \cap R$};
\draw [thick]  (0, 0.3) ellipse (1.8cm and 1.2cm);
\end{tikzpicture}
\caption{The first case: $\Omega \setminus R$ has one connected component.}
\end{figure}
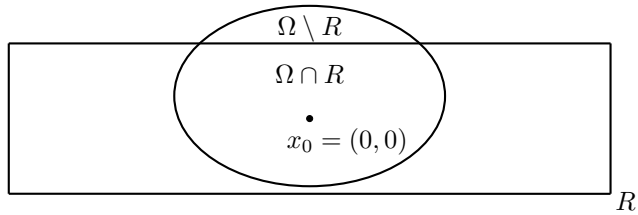
\end{center}

We start by looking at the first case, $\Omega \setminus R$ having one connected component, and show that it cannot occur: note that the function $v-u$ satisfies
\begin{align*}
 \Delta (v-u) &= 0 \qquad \mbox{on}~\Omega \cap R\\
v-u &\geq 0 \qquad \mbox{on}~\partial \left(\Omega \cap R\right)
\end{align*}
and, moreover, that $v-u>0$ on some subset of $\partial \left(\Omega \cap R\right)$. Moreover, by construction, $u(0,0) = v(0,0)$ and $\nabla (v-u)(0,0) = 0$, which implies
that there are at least 4 roots on the boundary and this is a contradiction.
We proceed to analyze the second case of $\Omega \setminus R$ having two connected components (and convexity of $\Omega$ implies that this is
an exhaustive case distinction).

\begin{center}
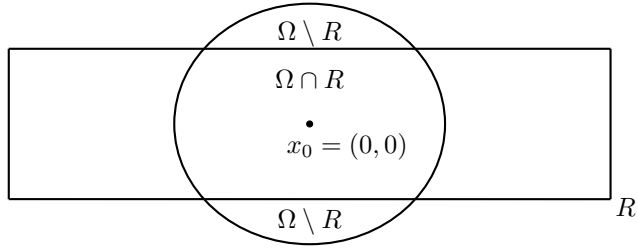
\begin{figure}[h!]
\begin{tikzpicture}[scale = 1]
\filldraw (0,0) circle (0.04cm);
\node at (0.5,-0.3) {$x_0=(0,0)$};
\draw [ thick] (-4, 1) -- (4,1.0001);
\draw [ thick] (-4, -1) -- (4,-1.001);
\draw [ thick] (-4, -1) -- (-4,1.0001);
\draw [ thick] (4, -1) -- (4,1.0001);
\node at (4.2, -1.1) {$R$};
\node at (0, 1.2) {$\Omega \setminus R$};
\node at (0, -1.3) {$\Omega \setminus R$};
\node at (0, 0.6) {$\Omega \cap R$};
\draw [thick]  (0, 0) ellipse (1.8cm and 1.6cm);
\end{tikzpicture}
\caption{The second case: $\Omega \setminus R$ has two connected components.}
\end{figure}
\end{center}

The second case can actually occur. We note again that
$$ \Delta (v-u) = 0 \qquad \mbox{on}~\Omega \cap R$$
and that $v-u$ is therefore harmonic. It has exactly 4 roots on the boundary, which implies that it may vanish in $(0,0)$ simultaneously with its derivatives but
cannot have the Hessian vanish as well. Moreover, we note that both Hessians of both functions are diagonal. 
We will now consider the Riemann mapping
$$\phi : \mathbb{D} \rightarrow \Omega \cap R$$ 
that sends the origin to the point $x_0$ in which the maximum is assumed. By construction, $(v-u) \circ \phi$ is a harmonic function in the unit disk. The
Koebe Quarter Theorem implies that $\phi: \mathbb{D} \rightarrow \Omega \cap R$ interpreted as an analytic function satisfies $|\phi'(0)| \leq 4$.
This implies that the second order derivatives of $(v-u)$ and $(v-u) \circ \phi$ are comparable in the origin. 
We have a fairly good understanding of the values of $v-u$ on the shorter sides (where it is simply given by $v$): up to small errors (that one could make precise), it
behaves roughly like $1-y^2$. In particular, this implies that
$$ \|v-u\|_{L^1(\partial \Omega)} \gtrsim 1.$$


It is therefore required to understand the amount to which distortion of these areas can happen. Classical estimates on harmonic measure then imply that
this distortion is at most exponential in $\diam/\inrad$ and this implies
$$ \|(v-u) \circ \phi\|_{L^1(\partial \mathbb{D})} \gtrsim \exp\left( - \beta \frac{\diam(\Omega)}{\inrad(\Omega)}\right) \qquad \mbox{for some universal}~\beta > 0.$$
We can now employ Corollary 3 (with $n=2$) and derive the existence of a direction $\nu$ such that
$$ \left| \frac{\partial^2}{\partial \nu^2} \left[ (v-u)\circ \phi \right] (0,0) \right| \gtrsim \frac{ \|(v-u) \circ \phi\|^5_{L^1(\partial \mathbb{D})} }{\|(v-u) \circ \phi\|^4_{L^{\infty}(\partial \mathbb{D})}}  \gtrsim \exp\left( - 5\beta \frac{\diam(\Omega)}{\inrad(\Omega)}\right).$$
Combining this with the second order derivatives of $v$ (derived above) implies the result.

\end{proof}

\section{Proof of the Propositions}
\subsection{Proof of Proposition 1}

This argument is much simpler and relies on a series of existing results based around the Makar-Limanov function. Since Makar-Limanov's original paper
\cite{makar} is not easy to find, we explicitly point out the papers of Henrot, Lucardesi \& Philippin \cite{henrot} and Keady \& McNabb \cite{keady} who describe the original approach.
\begin{proof}
Let
\begin{align*}
 -\Delta u &=1 \qquad \mbox{in}~\Omega  \\
u &= 0 \qquad \mbox{on}~\partial \Omega.
\end{align*}
We employ the Makar-Limanov function $P(u)$
$$ P(u) = \left\langle \nabla u, (D^2u) \nabla u\right\rangle - |\nabla u|^2 \Delta u + u\left( (\Delta u)^2 - D^2u \cdot D^2u\right),$$
where $\cdot$ denotes the Hadamard product. $P$ is superharmonic and therefore assumes its minimum on the boundary
$$P \geq \min P\big|_{\partial \Omega}.$$
A standard bound for this quantity is given by Payne \& Philippin \cite{payne2}
$$ P\big|_{\partial \Omega} \geq \frac{1}{8}\frac{ \min_{\partial \Omega}{\kappa} }{ \max_{\partial \Omega}{\kappa^3}},$$
where $\kappa$ is the curvature of the boundary of $\Omega$. Let us now assume w.l.o.g. that $x_0 = (0,0) \in \mathbb{R}^2$ and that
a local Taylor expansion (after possibly rotating the domain) is given by
$$ u(x,y) = \|u\|_{L^{\infty}(\Omega)} - \frac{a}{2} x^2 -  \frac{1-a}{2}y^2 + \mbox{l.o.t.}$$
The Makar-Limanov function in $x_0$ simplifies to
\begin{align*}
 P(u)\big|_{x_0} &= u\left( (\Delta u)^2 - D^2u \cdot D^2u\right)  \big|_{x_0} \\
&= \|u\|_{L^{\infty}(\Omega)} \left(1 - (a^2 + (1-a)^2) \right)\\
&\sim \|u\|_{L^{\infty}(\Omega)} \min(a,1-a),
\end{align*}
where the last equivalence  is valid for $a$ close to either 0 or 1 (the only case of interest).
Then
$$ \min(a,1-a) \gtrsim  \frac{\min P\big|_{\partial \Omega}}{\|u\|_{L^{\infty}(\Omega)}} \gtrsim \frac{1}{\|u\|_{L^{\infty}(\Omega)}} \frac{ \min_{\partial \Omega}{\kappa} }{ \max_{\partial \Omega}{\kappa^3}}.$$
Finally, we use the classical estimate for convex domains that 
$$ \|u\|_{L^{\infty}(\Omega)} \sim \inrad(\Omega)^2$$
and this gives the result.
\end{proof}

\subsection{Proof of Proposition 2}
\begin{proof} Integration by part yields 
$$ |\widehat{f}(k)| \lesssim \frac{\|f'\|_{L^1(\mathbb{T})}}{|k|}.$$
We observe that
$$ \sum_{|k| \geq n}{| \widehat{f}(k) |^2 } \lesssim  \sum_{|k| \geq n}{  \frac{\|f'\|^2_{L^1(\mathbb{T})}}{k^2} } \sim  \frac{\|f'\|^2_{L^1(\mathbb{T})}}{n}$$
is less than $\|f\|^2_{L^2(\mathbb{T})}$ for $n \gtrsim \|f'\|^2_{L^1(\mathbb{T})}\|f\|^{-2}_{L^2(\mathbb{T})}$, which implies, for suitable $c>0$ that
$$  \sum_{|k| \leq  c\|f'\|^2_{L^2(\mathbb{T})}\|f\|^{-2}_{L^1(\mathbb{T})}}{| \widehat{f}(k) |^2 } \geq \frac{\|f\|_{L^2(\mathbb{T})} }{2},$$
which then implies the first result via
\begin{align*}
 \| \theta_1 *f  \|_{L^2(\mathbb{T})} &= \sum_{k \in \mathbb{Z}}{e^{-k^2 t} |\widehat{f}(k)|^2} \geq  \sum_{|k| \leq  c\|f'\|^2_{L^1(\mathbb{T})}\|f\|^{-2}_{L^2(\mathbb{T})}}{e^{-k^2 t} |\widehat{f}(k)|^2}\\
&\gtrsim \exp\left( -c \frac{\|f'\|^4_{L^1(\mathbb{T})}}{\|f\|^4_{L^2(\mathbb{T})}} \right) \sum_{|k| \leq  c\|f'\|^2_{L^1(\mathbb{T})}\|f\|^{-2}_{L^2(\mathbb{T})}}{|\widehat{f}(k)|^2} \\
&\gtrsim  \exp\left( -c \frac{\|f'\|^4_{L^1(\mathbb{T})}}{\|f\|^4_{L^2(\mathbb{T})}} \right)  \frac{\|f\|_{L^2(\mathbb{T})} }{2} .
\end{align*}
 The second result follows more immediately from
$$  \sum_{k \geq n}{| \widehat{f}(k) |^2 } \lesssim  \frac{1}{n^2}\sum_{|k| \geq n}{  |k|^2 |\widehat{f}(k)|^2 } \lesssim  \frac{\|f'\|^2_{L^2(\mathbb{T})}}{n^2}$$
and then proceeding as above.
\end{proof}


\begin{thebibliography}{}

\bibitem{rima} R. Alaifari, L. Pierce and S. Steinerberger, 
Lower bounds for the truncated Hilbert transform.
Rev. Mat. Iberoam. 32 (2016), no. 1, 23--56. 

\bibitem{alv} L. Alvarez, F. Guichard, P.-L. Lions and J.-M. Morel,  
Axioms and Fundamental Equations of Image Processing,  Arch. Rational Mech. Anal. 123 (1993) 199--257.

\bibitem{bers} L. Bers, Local behavior of solutions of general linear elliptic equations. Comm. Pure Appl. Math. 8 (1955), 473--496.

\bibitem{bras} H.  Brascamp and E. Lieb, On extensions of the Brunn-Minkowski and Prekopa-Leindler
theorems, including inequalities for log concave functions, and with an application to
the diffusion equation, J. Functional Analysis, 22(4) (1976), 366--389.

\bibitem{caf} L. Caffarelli and J. Spruck, Convexity properties of solutions to some classical variational problems, Comm. Partial Differ. Equations, 7 (1982), 1337--1379.

\bibitem{caf2} L. Caffarelli L. and A. Friedman, Convexity of solutions of some semilinear elliptic equations,
Duke Math. J., 52(1985), 431--455.

\bibitem{cha} S.-Y. A. Chang, X.-N. Ma and P. Yang, Principal curvature estimates for the convex level sets of semilinear elliptic equations. Discrete Contin. Dyn. Syst. 28 (2010), no. 3, 1151--1164. 

\bibitem{bers} L. Bers,  
Local behavior of solutions of general linear elliptic equations. Comm. Pure Appl. Math. 8 (1955), 473--496. 

\bibitem{erem1} A. Eremenko and D. Novikov, 
Oscillation of functions with a spectral gap. 
Proc. Natl. Acad. Sci. USA 101 (2004), no. 16, 5872--5873. 

\bibitem{erem2} A. Eremenko and D. Novikov, 
Oscillation of Fourier integrals with a spectral gap. 
J. Math. Pures Appl. (9) 83 (2004), no. 3, 313--365. 

\bibitem{gab} R. Gabriel, A result concerning convex level surfaces of 3-dimensional harmonic functions, J. London Math.Soc., 32 (1957), 286--294.


\bibitem{hamel} F. Hamel, N. Nadirashvili and Y. Sire, Yannick
Convexity of level sets for elliptic problems in convex domains or convex rings: two counterexamples. 
Amer. J. Math. 138 (2016), no. 2, 499--527. 

\bibitem{henrot} A. Henrot, I. Lucardesi and G. Philippin, On two functionals involving the maximum of the torsion function, arXiv:1702.01258

\bibitem{kawohl} B.Kawohl, Rearrangements and convexity of level sets in PDE, Lectures Notes in
Math., 1150, Springer-Verlag, Berlin, 1985

\bibitem{kawohl1} B. Kawohl,Variations on the $p-$Laplacian. Nonlinear elliptic partial differential equations, 35--46, Contemp. Math., 540, Amer. Math. Soc., Providence, RI, 2011

\bibitem{keady} G. Keady and A. McNabb, 
The elastic torsion problem: solutions in convex domains. 
New Zealand J. Math. 22 (1993), no. 2, 43--64. 

\bibitem{ken} A. U. Kennington, Power concavity and boundary value problems, Indiana Univ. Math. J. 34 3 (1985), 687--704

\bibitem{roy} R.R.Lederman and S. Steinerberger, Stability Estimates for Truncated Fourier and Laplace Transforms, to appear in Integral Equations and Operator Theory

\bibitem{lions} P.-L. Lions, Two geometrical properties of solutions of semilinear problems. 
Applicable Anal. 12 (1981), no. 4, 267--272. 

\bibitem{lo1} M. Longinetti, Convexity of the level lines of harmonic functions, Boll. Un. Mat. Ital. A 6 (1983), 71--75.

\bibitem{lo2} M. Longinetti, On minimal surfaces bounded by two convex curves in parallel planes, J. Diff. Equations, 67 (1987), 344--358.

\bibitem{makar} L.G. Makar-Limanov, The solution of the Dirichlet problem for the equation $\Delta u = −1$ in a convex region, Mat.
Zametki 9 (1971) 89--92 (Russian). English translation in Math. Notes 9 (1971), 52--53.

\bibitem{mayeye} X.-N. Ma, J. Ye, Y.-H. Ye,
Principal curvature estimates for the level sets of harmonic functions and minimal graphs in $\mathbb{R}^3$. Commun. Pure Appl. Anal. 10 (2011), no. 1, 225--243. 

\bibitem{or} M. Ortel and W. Schneider, Curvature of level curves of harmonic functions, Canad. Math. Bull., 26 (1983), 399--405.

\bibitem{payne2} L. E. Payne and G. A. Philippin, Some remarks on the problems of elastic torsion and of torsional creep,
Some Aspects of Mechanics of Continua, Part I, Jadavpur University, 1977, pp. 32--40.

\bibitem{payne} L. E. Payne and G. A. Philippin,
Isoperimetric inequalities in the torsion and clamped membrane problems for convex plane domains. 
SIAM J. Math. Anal. 14 (1983), no. 6, 1154--1162. 

\bibitem{rachh} M. Rachh and S. Steinerberger, On the location of maxima of solutions of Schroedinger's equation, to appear in Comm. Pure. Appl. Math

\bibitem{rosay} J.-P. Rosay and W. Rudin, 
A maximum principle for sums of subharmonic functions, and the convexity of level sets. 
Michigan Math. J. 36 (1989), no. 1, 95--111. 

\bibitem{saad} M. Sadd, Elasticity, Third Edition: Theory, Applications, and Numerics, Academic Press (2014)

\bibitem{weinkove} B. Weinkove, Convexity of level sets and a two-point function, arXiv:1701.05820

\bibitem{sturm} A. Zettl, Sturm-Liouville theory.  Mathematical Surveys and Monographs, 121. American Mathematical Society, Providence, RI, 2005. 
\end{thebibliography}
\end{document}